\documentclass[10pt]{article}

\usepackage{amsmath}
\usepackage{amssymb}
\usepackage{amscd}
\begin{document}

 \title {\textbf{Orthogonal Projections Based on Hyperbolic and Spherical $n$-Simplex}}

\author{ Baki Karl\i\v{g}a$^1$,~Murat Sava\d{s}$^1$~and ~{Atakan T. Yakut}$^2$ \\ {\footnotesize $^1$ Gazi University,  Faculty of Science, Department of Mathematics~~~~~~~~~~~~} \\ {\footnotesize 06500, Ankara-Turkey~~~~~~~~~~~~~~~~~~~~~~~~~~~~~~~~~~~~~~~~~~~~} \\ {\footnotesize $^2$ Nigde University, Art and Sciences Faculty, Department of Mathematics} \\ {\footnotesize ~51100,~Nigde-Turkey~~~~~~~~~~~~~~~~~~~~~~~~~~~~~~~~~~~~~~~~~~~~~~}
\\{\noindent\footnotesize e-mail: karliaga@gazi.edu.tr}, {\footnotesize msavas@gazi.edu.tr}, {\footnotesize sevaty@nigde.edu.tr}}
\date{}

\maketitle
\newtheorem {thm}{Theorem}[section]
\newtheorem{lem}[thm]{Lemma}
\newtheorem{prop}[thm]{Proposition}
\newtheorem{cor}[thm]{Corollary}
\newtheorem{df}[thm]{Definition}
\newtheorem{nota}{Notation}
\newtheorem{note}[thm]{Remark}
\newtheorem{ex}[thm]{Example}
\newtheorem{exs}[thm]{Examples}
\newtheorem{rmk}[thm]{Remark}
\newtheorem{quo}[thm]{Question}
\newenvironment{proof}{\par\noindent{\bf Proof \,}}{$\hfill \Box$\par\bigskip}

{\begin{center}
 {\bf 2010 MSC:} 51M09, 51M10, 51M20, 51M25, 52A38, 52A55
\end{center}}

{\bf Key words:} Projection, hyperbolic, spherical, simplex, Gram matrix, edge matrix

\begin{abstract}
In this paper, orthogonal projection along a geodesic to the chosen $k-$plane is introduced using edge and Gram matrix of an $n-$simplex in hyperbolic or spherical $n-$space.
The distance from a point to $k-$plane is obtained by the orthogonal projection. It is also given the perpendicular foots from a point to $k-$plane of hyperbolic and spherical $n-$space.

\end{abstract}

\section{Introduction}

\hspace{0,5cm} One of the fundamental notions in geometry is orthogonal projection and also studies extensively
through the long history of mathematics and physics.
There are many applications of orthogonal projection. The concept of orthogonal projection plays an important role in the scattering theory, the theory of many-body resonance and different branches of theoretical and mathematical physics.

\vspace{0,5cm}
Let $R_{1}^{n+1}$ be  $(n+1)-$dimensional vector space
equipped with the scalar product $\langle,\rangle$ which is defined by
\[
\langle x,y\rangle=-x_{1}y_{1}+\sum\limits_{i=2}^{n+1}x_{i}y_{i} .
\]
 \noindent If the restriction of scalar product on a
subspace $W$ of ${R_{1}}^{n+1}$ is positive definite, then the
subspace $W$ is called {\it space-like}; if it is positive
semi-definite and degenerate, then $W$ is called {\it light-like};
if $W$ contains a time-like vector of ${R_{1}}^{n+1}$, then $W$ is
called {\it time-like}.

\vspace{0.2cm}

${S}_{1}^{n}~=~\{x\in {R_{1}}^{n+1}\mid \langle x,x\rangle=1\}$ is called {\it de Sitter n-space}. The $n-$dimensional {\it unit pseudo-hyperbolic space} is
defined as
\[
H_{0}^{n}~=~\left\{ x\in {R_{1}}^{n+1}\mid \langle x,x\rangle=-1\right\},
\]
\noindent which has two connected components, each of which can be
considered as a model for the {\it n-dimensional hyperbolic space}
${H}^n$.~Throughout this paper we consider {\it hyperbolic $n-$space}
\[
{\ H}^n =\{x\in H_{0}^{n}\mid x_{1}> 0\}.
\]
\noindent Hence, each pair of points $p_{i},p_{j}$ in ${ H}^n$
satisfy $\langle p_{i},p_{j}\rangle$ $< 0$. The{ \it hyperbolic distance} for $p, q \in H^n$ is defined by
$arccosh(-\langle p,q\rangle)$. Since each $e\in S_1^n$ determines a time-like hyperplane of $R_1^{n+1}$, we have hyperplane $e^{\bot }\cap H^{n}$ of $H^{n}$.

\vspace{0.2cm}
\noindent Let $R^{n+1}$ be  $(n+1)-$dimensional vector space
equipped with the scalar product ${\langle,\rangle}_E$ which is defined by
\[
\langle x,y\rangle_E=\sum\limits_{i=1}^{n+1}x_{i}y_{i} .
\]
The $n$-dimensional {\it unit spherical space} $S^n$ is given by
$$S^n=\{x\in {R^{n+1}}|\langle x,x\rangle_E=1\}.$$
\noindent The spherical distance $d_s(p,q)$ between $p$ and $q$ is given by $arccos(\langle p,q\rangle_E).$

\vspace{0.2cm}
\noindent We consider $W$ is a vector subspace spanned by the vectors
$e_{1}, e_{2}, ..., e_{n-k}$ in ${ S}_1^n$. By using Lemma 27 of
\cite{On}, one can easily see that $W$ is $(n-k)-$dimensional
time-like subspace and $V= e_{1}^{\bot }\cap e_{2}^{\bot }\cap
...\cap e_{n-k}^{\bot }$ is $(k+1)-$dimensional time-like
subspace of $R_{1}^{n+1}$. Consequently, for $i=1, 2, ..., n-k$, the
hyperplane $e_{i}^{\bot }\cap{H}^{n}$ intersect at the
time-like $k-$plane $V\cap{H}^{n}$ of $H^n$. One can define the same tools for spherical $n-$space.

 \vspace{0.2cm}
\noindent Let $\bigtriangleup $ be  a hyperbolic or spherical $n-$simplex with
vertices $p_{1}, ..., p_{n+1}$, and $\bigtriangleup_{i}$ be the face
opposite to vertex $p_{i}$. Then, according to the first
section of \cite{KY}, we have the edge matrix $M$ and Gram matrix $G$
of $\bigtriangleup $. Let $|M|$ and $M_{ij}$ be the determinant and $ijth-$minor of $M$, then  the unit outer normal vector of $\bigtriangleup_{i}$ is given by
$$
e_{i}= \frac{-  \epsilon\displaystyle{\sum_{j=1 }^{n+1} {M_{ij}p
_{j}}}}{\sqrt{{M_{ii}}|M|}} ,~~~i=1,...,n+1,
$$
where $ \epsilon$ is the curvature of space.

 \vspace{0.2cm}
The intersection of $ H^{n}$ with
$(k+1)-$dimensional time-like subspace is called $k-$dimensional
plane of $ H^{n}$ \cite{IT}.
Similarly, a $k-$plane of spherical space is given by the same way.

 \vspace{0.2cm}
When a geodesic is drawn
orthogonally from a point to a $k-$plane, its intersection with
the $k-$plane is known as {\it perpendicular foot} on that $k-$plane in ${H}^n$ or ${S}^{n}$ .
The length of a geodesic segment bounded by a point and its perpendicular foot is called
{\it the distance between that point and $k-$plane}. The distance between a vertex and its any opposite $k-$face is called {\it $k-$face altitude} of an $n-$simplex.

 \vspace{0.2cm}
The orthogonal projection to $2-$plane in Euclidean
space is well-known (see \cite{Re},\cite{Ar},\cite{Br},\cite{BGT}). The orthogonal projection to $k-$plane in
Euclidean space is given in \cite{Ka}. The orthogonal projection
taking a point in $H^n$ and mapping it to its
perpendicular foot on a hyperplane are studied in \cite{IT} and
\cite{AKS}, respectively. The distance between a point and a
hyperbolic(spherical) hyperplane is introduced in \cite{Vi}. The altitude of $(n-1)-$face of hyperbolic $n-$simplex is given
in \cite{DMP}.

\vspace{0.2cm}
The orthogonal projection taking a point along a geodesic and mapping to its
perpendicular foot, where geodesic meets orthogonally the chosen
$k-$plane of projection, has not been studied. The aim
of this paper is to study such orthogonal projections according to the edge
matrix of a simplex in $H^{n}$ or $S^{n}$.

\vspace{0.2cm}
Let $m^{k+1}$  be the determinant of sub-matrix $M( k+1,...,n+1)$ of $M$ and $g^{k+1}$ be the determinant of
sub-matrix $ G( k+1,...,n+1)$ of $G$. Suppose that ${m}_{i}^{j}$ and ${g}_{i}^{j}$ be the
determinant of sub-matrix ${M \left(\begin{array}{cccc}
1&...&k+1&i
\\ 1&...&k+1&j \end{array}\right)}_{
i,j=k+2, ..., n+1}$ and ${G
\left( \begin{array}{cccc}1&...& k+1&i\\
1&...&k+1&j\end{array} \right)}_{
i,j=k+2, ..., n+1},$ respectively.

\vspace{0.2cm}
\begin{lem}\label{1}
 \noindent Let $\bigtriangleup$ be a hyperbolic or spherical $n-$simplex with the edge matrix $M$ and
Gram matrix $G$. Let $M_{ii}$ and $G_{ii}$ be
$i$th minor of $M$ and  $G$, respectively. Then
 $ M^{-1} = T G T$  and $G^{-1} =TMT$
 \vspace{0.2cm}
 \\where  $ T = \displaystyle\left[ \sqrt\frac{G_{ii}}{\epsilon|G| }~\delta _{ij}\right] =
\displaystyle \left[ \sqrt\frac{M_{ii}}{\epsilon|M| }~\delta
_{ij}\right]$.
\end{lem}
\begin{proof}
It can be seen from \cite{K}.
\end{proof}

Let $M^{11}, M^{12}, M^{22}$ and $G^{11}, G^{12}, G^{22}$ be $(k+1)$x$(k+1)$, $(k+1)$x$(n-k)$, $(n-k)$x$(n-k)$ types sub-matrices of $M$ and $G$, respectively.
Suppose that $M, G$ and diagonal matrix $T$ partitioned as $\left[\begin{array}{cc} M^{11}&M^{12}\\M^{12}&M^{22}
\end{array}\right], \left[\begin{array}{cc}
G^{11}&G^{12}\\G^{12}&G^{22} \end{array}\right]$ and
 $
\left[\begin{array}{cc} T^{11}&0\\0&T^{22} \end{array}\right]$, respectively.

\vspace{0.2cm}
\noindent Concerning Lemma~\ref{1} along with  Schur complement of a
symmetric matrix, we have the following lemma.

\begin{lem}\label{2}
\noindent Let $ \displaystyle S_ {M^{ii}}$ and
$\displaystyle S_ {G^{ii}}$ be Schur complements of the
sub-matrices $ M^{ii}$ and $ G^{ii}$. Then
$$
\left(M^{ii}\right)^{-1}=T^{ii}\displaystyle
S_{G^{jj}}T^{ii},~
\left(G^{ii}\right)^{-1}=T^{ii}\displaystyle
S_{M^{jj}}T^{ii} ~~~i \neq j; ~~i,j=1,2.
$$
\end{lem}

\begin{proof}
\noindent It is obvious that $M, G$ are symmetric and $M^{ii}, G^{ii}$ are
invertible. Since the inverse of Schur complement of $M^{11}$ in $M$ is the sub-matrix of $M^{-1}$, we have
\footnotesize{
\[
\begin{array}{c}
M^{-1}=\left[
\begin{array}{ccc}
\left( M^{11}\right) ^{-1}+\left( M^{11}\right) ^{-1}M^{12}{(S_{M^{11}})^{-1}}M^{21}\left( M^{11}\right) ^{-1}& &-\left( M^{11}\right) ^{-1}M^{12}{(S_{M^{11}})^{-1}} \\
& &  \\
-{(S_{M^{11}})^{-1}}M^{21}\left( M^{11}\right) ^{-1}& &{\displaystyle(S_{M^{11}})^{-1}}\end{array}\right]  \end{array}\].}

\noindent \normalsize Similarly for the Schur complement of $M^{22}$, we obtain

\footnotesize{\[
\begin{array}{c}
M^{-1}=\left[
\begin{array}{ccc}
{\displaystyle(S_{M^{22}})^{-1}} &  & -{(S_{M^{22}})^{-1}}M^{12}\left(
M^{22}\right) ^{-1} \\
&  &  \\
-\left( M^{22}\right) ^{-1}M^{21}{(S_{M^{22}})^{-1}} &  & \left(
M^{22}\right) ^{-1}+\left( M^{22}\right) ^{-1}M^{21}{(S_{M^{22}})^{-1}}M^{12}\left( M^{22}\right) ^{-1}\end{array}\right]  \end{array}\]}

\noindent \normalsize Then we have \\

$\begin{array}{ccc}
M^{-1}=\left[
\begin{array}{cc}
{\displaystyle(S_{M^{22}})^{-1}} & -\left( M^{11}\right) ^{-1}M^{12}{(S_{M^{11}})^{-1}} \\
&  \\
-\left( M^{22}\right) ^{-1}M^{21}{(S_{M^{22}})^{-1}} & {\displaystyle(S_{M^{11}})^{-1}}\end{array}\right]\end{array},$\\

\noindent by the same way, we get \\

$\begin{array}{ccc}
G^{-1}=\left[
\begin{array}{cc}
{\displaystyle(S_{G^{22}})^{-1}} & -\left( G^{11}\right) ^{-1}G^{12}{(S_{G^{11}})^{-1}} \\
&  \\
-\left( G^{22}\right) ^{-1}G^{21}{(S_{G^{22}})^{-1}} & {\displaystyle(S_{G^{11}})^{-1}}\end{array}\right]\end{array}.$\\

\noindent Thus, we obtain the desired results  using Lemma~\ref{1}.
\end{proof}

\section{\textbf{Orthogonal Projection to $k-$plane of $H^{n}$ \\ based on a Hyperbolic $n-$simplex}}

If $p_{1}, p_{2}, ..., p_{n+1}$ are vertices of any hyperbolic $n-$ simplex $\bigtriangleup $, then $\{p_{1}, p_{2}, ..., p_{n+1}\}$ is
a basis of $R_{1}^{n+1}$. Let $W_j$ be a subspace spanned by $\{\footnotesize{p_{1},p_{2},...,\hat p_{j},...,p_{n+1}}\}$, and $e_j$ be the unit outer normal to
$W_j$; $j =1,...,n+1$. Hence $\{e_{1},e_{2},...,e_{n+1}\}$ is
another basis of $R_{1}^{n+1}$.

\vspace{0.2cm}
Let $W^h$ be a $k-$plane which contains a $k-$face with vertices $ p_{1}, p_{2}, ..., p_{k+1}$
of $\bigtriangleup$. Then the set $\{p_{1}, p_{2}, ..., p_{k+1}\}$ is
a basis of the $(k+1)-$dimensional subspace of $W$ in $
R_{1}^{n+1}$. Since $p_{1}, p_{2}, ..., p_{k+1}$ are vertices of
$\bigtriangleup$, the subset $\{p_{1}, p_{2}, ..., p_{k+1}\}$ of  $R_{1}^{n+1}$  can be extended basis
$\{p_{1}, p_{2}, ..., p_{n+1}\}$  and $\{p_{1}, ..., p_{k+1}, e_{k+2}, ..., e_{n+1}\}$ of  $R_{1}^{n+1}$.
As a consequence, we see that $\{e_{k+2}, ..., e_{n+1}\}$ is a basis of \\$(n-k)-$dimensional subspace $W^{\bot}$.

\begin{thm}\label{th3}
 \noindent Let  $p$ be a point and $W^h$ be a $k-$plane in $H^n$. Then the orthogonal projection
 $\sigma(p)$ of $p$ to $W^h$  is given by
$$
\sigma(p)=\frac{p+\displaystyle\sum\limits_{{s,t=k+2}}^{n+1}{\frac{\sqrt{M_{ss}M_{tt}}
m_{t}^{s} \langle p,e_t\rangle {e_s}} {|M|m^{k+1}}}} {\sqrt{1-
\displaystyle\sum\limits_{{s,t=k+2}}^{n+1}
{\displaystyle\frac{\sqrt{M_{ss}M_{tt}} m_{t}^{s} \langle p,{e_t} \rangle
\langle p,{e_s}\rangle }{|M|m^{k+1}}}}}.
$$
\end{thm}

\begin{proof}
\noindent For a point $p\in H^n$, by \cite[ Theorem 3.11] {Vi},
there is a point $\dot p \in W$ such that $\vec{p\dot p}\in
W^{\bot }$. Therefore, we can write

\begin{equation}
 \overrightarrow{p\dot p}=\sum\limits_{{s=k+2}}^{n+1}\lambda _s e_s.
\end{equation}

\noindent Then, we have
$ -\langle p,{e_t}\rangle=\displaystyle\sum\limits_{{s=k+2}}^{n+1}\lambda _s \langle
{e_s} ,{e_t}\rangle$, ~~~$t=k+2, ..., n+1$.

\noindent Taking
$$
G^{22}=G\left(\begin{array}{cccc}k+2&...& n+1 \\ k+2&...&n+1\end{array}\right)
=\left[\langle
{e_s} ,{e_t}\rangle \right]_{s,t=k+2,...,n+1},  ~~L=\left[\lambda _{k+2}
... \lambda_{n+1} \right],
$$
\noindent we obtain

\begin{equation}
 L= -\left(G^{22}\right)^{-1}\left[\langle p,{e_t}\rangle \right].
\end{equation}

\noindent By Lemma~\ref{2} and the equation (5) of \cite{BS}, we
see that
$$
\left(G^{22}\right)^{-1}=\left[\displaystyle\frac{\sqrt{M_{ii}
M_{jj}}m_{j}^{i}}{-|M|m^{k+1}}\right]_{i,j=k+2,...,n+1},
$$
\noindent and this implies

 \begin{equation}
\lambda _s
=\displaystyle\sum\limits_{{t=k+2}}^{n+1}\frac{\sqrt{M_{ss}M_{tt}}
m_{t}^{s} \langle p,{e_t}\rangle}{|M|m^{k+1}},~~~s=k+2,...,n+1.
\end{equation}

 \noindent Substituting (3) into (1), we obtain

\begin{equation}
\dot p=
p+\displaystyle\sum\limits_{{s,t=k+2}}^{n+1}{\displaystyle\frac{\sqrt{M_{ss}M_{tt}}
m_{t}^{s} \langle p,{e_t}\rangle {e_s}}{|M|m^{k+1}}}.
\end{equation}

\noindent By \cite{Vi}, there exists a unique $\sigma(p)\in W^h$ such that
$\sigma(p)=c\dot p$. Since  $\dot p$ is the orthogonal projection
of $p$ to $W$, we have
$$
c=\displaystyle\frac{1}{\sqrt{1-\displaystyle
\sum\limits_{{s,t=k+2}}^{n+1}
{\displaystyle\frac{\sqrt{M_{ss}M_{tt}} m_{t}^{s} \langle p,{e_t}\rangle
\langle p,{e_s}\rangle }{|M|m^{k+1}}}}}
$$
\noindent which completes the proof.
\end{proof}

In case of the orthogonal projection to a hyperplane, we obtain $m_{j}^{j}=|M|$ and
$m^{k+1}=M_{jj}$. Substituting these equalities into the statement
of Theorem~\ref{th3}, we reach  the result of \cite[Theorem 4.1 ]
{IT} and \cite[Proposition 2.2] {Us}, as follows:

$$
\sigma(p)=\displaystyle \frac{p- \langle p,{e_j}\rangle e_j} {\sqrt{1+ \langle p,{e_j}\rangle^2
}}, ~~j=1,...,n+1
$$
\noindent where $e_j$ is the unit normal of ${W_j}$ in $R_1^{n+1}$. This result is also a generalization of Theorem \ref{th3} \cite{{IT}, {Us}}.

\begin{thm}\label{th5}
\noindent Let  $p$ be a point and $W^h$  be a $k-$plane in $H^n$. Then,
$$
\cosh \xi (p,W^h) =\displaystyle{\sqrt{1-
\sum\limits_{{s,t=k+2}}^{n+1}
{\displaystyle\frac{\sqrt{M_{ss}M_{tt}} m_{t}^{s} \langle p,{e_t}\rangle
\langle p,{e_s}\rangle }{|M|m^{k+1}}}}}.
$$

\noindent where  $\xi (p, W^h)$ is the distance between $p$ and
$W^h$.
\end{thm}

\begin{proof}
\noindent Since $\langle p, \sigma (p)\rangle=-\cosh \xi (p, W^h) $, the result
follows Theorem~\ref{th3}.
\end{proof}

As an immediate consequence of
Theorem~\ref{th5}, we obtain the following known
result\cite[Section 4] {Vi}.

\begin{cor}
\noindent Let $p$ be a point and ${W_j^h}$ be a hyperplane of $H^n$ determined by $e_j$. Then the distance
$\xi(p,{W_j^h})$ between $p$ and ${W_j^h}$ is given by

$$\cosh\xi(p,{W_j^h})=\sqrt{1+ \langle p,{e_j}\rangle^2 }.$$
\end{cor}

By taking $p_j$ instead of $ p$
in (4) and using $\langle p_j,{e_t}\rangle=-\sqrt{\displaystyle\frac{|M|}{
M_{tt}}}\delta_{jt},$

\noindent we obtain
$$
 \dot
p_j=p_j+\displaystyle\sum\limits_{{s=k+2}}^{n+1}\sqrt{\displaystyle\frac{M_{ss}}{|M|}}\displaystyle\frac{m_{j}^{s}}{m^{k+1}} e_s
$$
and
\noindent
$$
\noindent
 \langle \dot p_j,\dot p_j\rangle =-1-\displaystyle\frac{m_{j}^{j}}{m^{k+1}}
$$

\noindent where $p_j$ is a vertex of $\bigtriangleup$. The proof of  following corollary is obvious from
Theorem~\ref{th3}.

\begin{cor}
\noindent Let $\bigtriangleup$ be a hyperbolic simplex with
vertices $p_{1},...,p_{n+1}$. Then the perpendicular foot from
$p_j$ to $k-$face $W^h$ is
given by
$$
\sigma(p_j)=\frac{p_j+\displaystyle\sum\limits_{{s=k+2}}^{n+1}
\sqrt{\frac{M_{ss} } { |M| }}{\displaystyle\frac{~m_{j}^{s}}{m^{k+1}}}~{e_s}}
{\sqrt{1+\displaystyle\frac{m_{j}^{j}}{m^{k+1}}}}, ~~~j=k+2,...,n+1,
$$
where $p_{1},...,p _{k+1}$ are vertices of $k-$face $W^h$.
\end{cor}

 \noindent\vspace{0.5cm} If we replace $p$ by $ p_j$ and use
 $\langle {p_j},{e_t}\rangle=-\sqrt{\displaystyle\frac{ \left|M\right| }{
M_{tt}}}\delta_{jt}$, we see that
$$
\langle \sigma(p_j), {p_j}\rangle=
-\sqrt{1+\displaystyle\frac{m_{j}^{j}}{m^{k+1}}}.
$$
 \noindent If we consider the last equation in the proof of Theorem~\ref{th3}, we see that $
\cosh \xi (p_j,W^h) =\displaystyle\sqrt{1+ \displaystyle\frac{
m_{j}^{j}}{m^{k+1}}},
$ that result is a generalization of \cite[Proposition 4]
{DMP} to the $k-$face $W^h$ of a hyperbolic $n-$simplex. Since $\displaystyle\frac{~m_{j}^{j}}{m^{k+1}}$ is the
diagonal $jjth-$entry of $ \displaystyle
S_{M^{11}}=\left[a_{ij}\right]$, the altitude from $p_j$ to $k-$face $W^h$ with vertices $p_{1},...,p_{k+1}$ is given by

$$
\cosh \xi (p_j,W^h) =\displaystyle\sqrt{1+ a_{jj}}
$$
\noindent where  $\xi (p_j, W^h)$ is the distance between $p_j$ and
$k-$ face $W^h$.

\vspace{0.2cm}
 By $\dot p_j = p_j+\displaystyle \sqrt{\frac{|M|
}{M_{jj}}}~{e_j}$, for $(n-1)-$face $W_j^h$, we have
the following corollary.

\begin{cor}
\noindent Let $\bigtriangleup $ be a hyperbolic simplex with
vertices $p_{1},...,p_{n+1}$. Then the perpendicular foot from
$p_j$ to $(n-1)-$face $W_j^h$ is given by
$$
\sigma(p_j)=\frac{p_j+\displaystyle \sqrt{\frac{|M|
}{M_{jj}}}~{e_j}} {\sqrt{1+ \displaystyle\frac{|M|
}{M_{jj}}}}, ~~~j=1,...,n+1
$$

where $p_{1},...,\hat p_j,...,p_{n+1}$ are vertices of $W_j^h$.
\end{cor}

\vspace{0.2cm}
Using $ G_{jj}=\displaystyle\frac{-|G|
M_{jj}}{|M|}$ for $j=1,...,n+1$, we obtain the following
known result\cite[Proposition 4] {DMP}.

\begin{cor} Let $\bigtriangleup $ be a hyperbolic simplex with
vertices $p_{1},...,p_{n+1}$. Then the altitude $\xi (p_j, W_j^h)$ from $p_j$ to
$(n-1)-$face $W_j^h$  is given by
$$
\cosh \xi (p_j, W_j^h) =\displaystyle{\sqrt{1+
\displaystyle\frac{|M|}{M_{jj}}}}, ~~~j=1,...,n+1
$$

\noindent where $p_{1},...,\hat
p_{j},...,p_{n+1}$ are vertices of $W_j^h$.
\end{cor}

\section{\textbf{Orthogonal Projections to a $k-$plane of $S^{n} $ \\ based on a Spherical $n-$simplex }}

Let $\bigtriangleup $ be  with vertices
$p_{1},...,p_{n+1}$. Then $\{p_{1},...,p_{n+1}\}$ is  a basis of
$R^{n+1}$. If $W_j$ is the subspace spanned by
 $\{p_{1},...,\hat p_{j},...,p_{n+1}\}$, then $\{e_{1},...,e_{n+1}\}$ is  another basis of $R^{n+1}$ where $e_j$  is the unit outer normal to
$W_j$  for $j=1,...,n+1$.

\vspace{0.2cm}
\noindent Let $W^s$ be a $k-$plane which contains a $k-$face with vertices
$p_{1},p_{2},...,p_{k+1}$. Then the  set
$\{p_{1},p_{2},...,p_{k+1}\}$  is a basis of the $(k+1)-$dimensional subspace $W$ in $R^{n+1}$.
As a consequence, we have a basis $\{e_{k+2}...,e_{n+1}\}$ of $(n-k)-$dimensional
subspace $ W^{\bot }$.

\begin{thm}\label{th23}
 \noindent Let  $p $ be a point and $W^s$ be a $k-$plane in $S^n$. Then the orthogonal projection
 $\sigma(p)$ of $p$ to $W^s$  is given by
$$
\sigma(p)=\frac{p-\displaystyle\sum\limits_{{s,t=k+2}}^{n+1}{\frac{\sqrt{M_{ss}M_{tt}}
m_{t}^{s} {\langle p,{e_t}\rangle} _E{e_s}} {|M|m^{k+1}}}} {\sqrt{1-
\displaystyle\sum\limits_{{s,t=k+2}}^{n+1}
{\displaystyle\frac{\sqrt{M_{ss}M_{tt}} m_{t}^{s} {\langle p,{e_t}\rangle}_E
{\langle p,{e_s}\rangle} _E}{|M|m^{k+1}}}}}.
$$
\end{thm}

\begin{proof}
By \cite[Theorem 3.11] {Vi}, for $ p\in S^n$, there is a  $\dot p \in W$ such that
$\vec{p\dot p}\in W^{\bot }$. Therefore, we can write

\begin{equation}
 \overrightarrow{p\dot p}=\sum\limits_{{s=k+2}}^{n+1}\lambda _s e_s
\end{equation}

\noindent Then, we have

$ {\langle p,\overrightarrow{p\dot p}\rangle}_E=\sum\limits_{{s=k+2}}^{n+1}\lambda
_s {\langle e_s ,{e_t}\rangle}_E  .$

 \noindent Taking
$$
 G^{22}=G \left(\begin{array}{cccc}k+2&...& n+1 \\ k+2&...&n+1\end{array}\right)
=\left[{\langle {e_s},{e_t}\rangle}_E\right]_{s,t=k+2,...,n+1}, ~~L=\left[\lambda _{k+2}
... \lambda_{n+1} \right],
$$
\noindent we find

\begin{equation}
 L= -\left(G^{22}\right)^{-1}\left[{\langle p,{e_t}\rangle}_E\right].
\end{equation}

\noindent By Lemma~\ref{2} and the equation (5) of \cite{BS}, we see that
$$
\left(G^{22}\right)^{-1}=\left[\displaystyle\frac{\sqrt{M_{ii}
M_{jj}}m_{j}^{i}}{|M|m^{k+1}}\right]_{i,j=k+2,...,n+1}.
$$
\noindent This implies

\begin{equation}
\lambda _s
=-\displaystyle\sum\limits_{{t=k+2}}^{n+1}\frac{\sqrt{M_{ss}M_{tt}}
m_{t}^{s} {\langle p,{e_t}\rangle}_E}{|M|m^{k+1}}, ~~~s=k+2,...,n+1.
\end{equation}

\noindent Substituting (7) into (5), we obtain

\begin{equation}
\dot p=
p-\displaystyle\sum\limits_{{s,t=k+2}}^{n+1}{\displaystyle\frac{\sqrt{M_{ss}M_{tt}}
m_{t}^{s} {\langle p,{e_t}\rangle}_E {e_s}}{|M|m^{k+1}}}.
\end{equation}

\noindent By \cite{Vi}, there exists a unique $\sigma(p)\in W^s $
such that $\sigma(p)=c\dot p$. Since  $\dot p$ is the orthogonal
projection of $p$ to $W$, we have
$$
c = \displaystyle\frac{1}{\sqrt{1-\displaystyle
\sum\limits_{{s,t=k+2}}^{n+1}
{\displaystyle\frac{\sqrt{M_{ss}M_{tt}} m_{t}^{s} {\langle p,{e_t}\rangle}_E
{\langle p,{e_s}\rangle}_E }{|M|m^{k+1}}}}}
$$
\noindent which  completes  the proof.
\end{proof}

By Theorem \ref{th23}, we have

$$
\sigma(p)=\displaystyle \frac{p-{\langle p,{e_j}\rangle} _E e_j} {\sqrt{1-{\langle p,{e_j}\rangle}_E^2
}}
$$

\noindent where $e_j$ is the unit normal of the $W_j$ in $R^{n+1}$.

\begin{thm}\label{th24}
Let  $p$  be a point and $W^s$ be a $k-$plane in $S^n$. Then
$$
\cos\theta (p,W^s) =\displaystyle{\sqrt{1-
\sum\limits_{{s,t=k+2}}^{n+1}
{\displaystyle\frac{\sqrt{M_{ss}M_{tt}} m_{t}^{s} {\langle p,{e_t}\rangle}_E
{\langle p,{e_s}\rangle}_E }{|M|m^{k+1}}}}}.
$$

\noindent where $\theta (p,W^s)$ is the distance between $p$ and $W^s$.
\end{thm}

By taking  $p_j$ instead of $p$ and using $ {\langle {p_j},{e_j}\rangle}_E=-\displaystyle \sqrt{\frac{|M| }{M_{jj}}}$ in (8), we obtain
$$
\dot p_j =
p_j+\displaystyle\sum\limits_{{s=k+2}}^{n+1}{\displaystyle\sqrt{\frac{M_{ss}}{|M| }}} {\displaystyle\frac{m_{j}^{s}}{m^{k+1}}}
{e_s}
$$
\noindent and
$${\langle \dot p_j,\dot p_j\rangle}_E=1-\displaystyle\frac{m_{j}^{j}}{m^{k+1}},~~~{ j=k+2,...,n+1},
$$

\noindent where $p_j$ is a vertex of $\bigtriangleup$. Hence, we have the following corollary.

\begin{cor}
\noindent Let $\bigtriangleup $ be a spherical $n-$simplex with
vertices $p_{1},...,p_{n+1}$, then the  perpendicular foot from
$p_j$ to $k-$face $W^s$  is
given by
$$
\sigma(p_j)=\frac{p_j+\displaystyle\sum\limits_{{s=k+2}}^{n+1}
\sqrt{\frac{M_{ss} }{|M|}}~~{\frac{~m_{j}^{s}}{m^{k+1}}}~{e_s}}
{\displaystyle\sqrt{1-\frac{ m_{j}^{j}}{m^{k+1}}}}, ~~~j=k+2,...,n+1,
$$
where $p_{1},...,p_{k+1}$ are vertices of $W^s$.
\end{cor}

\vspace{0.2cm} Let $\theta (p_j,W^s)$  be the
altitude from the vertex $p_j$ to the $k-$face $W^s$ with vertices
$p_{1},...,p_{k+1}$ for $j=k+2,...,n+1$. Then  $\theta (p_j,W^s)$ is given by

$$
\cos\theta (p_j,W^s) =\displaystyle\sqrt{1- \displaystyle\frac{
m_{j}^{j}}{m^{k+1}}}.
$$

\noindent By equality (5) in \cite{BS}, the $jjth-$entry of the Schur
complement $ \displaystyle S_{{M^{11}}}=\left[b_{ij}\right]$
satisfy $ \displaystyle
b_{jj}=\displaystyle\frac{m_{j}^{j}}{m^{k+1}}$.

Let $W_{j}^s $ be  the $(n-1)-$face with vertices $p_{1},...,\hat
p_{j},...,p_{n+1}$ of $\bigtriangleup$. Then, we have
$$
\dot p_j =p_j+\displaystyle \sqrt{\frac{\mid M\mid
}{M_{jj}}}~{e_j},
$$
 \noindent and
$$ {\langle {p_j},{e_j}\rangle}_E=-\displaystyle \sqrt{\frac{\mid M\mid
}{M_{jj}}}, ~~~j=1,...,n+1.
$$

The proof of the following corollary is obtained by using the above
equations.

\begin{cor}
Let $\bigtriangleup$ be a spherical simplex with vertices
$p_{1},...,p_{n+1}$. Then the  perpendicular foot from $p_j$ to
$(n-1)-$face $W_j^s$  is given by
$$
\sigma(p_j)=\frac{p_j+\displaystyle \sqrt{\frac{|M|}{M_{jj}}}~{e_j}} {\sqrt{1- \displaystyle\frac{|M|}{M_{jj}}}}, ~~~j=1,...,n+1,
$$
where $p_{1},...,\hat p_j,...,p_{n+1}$ are vertices of $W_j^s$.
\end{cor}

\begin{cor} Let $\bigtriangleup$ be a spherical simplex with vertices
$p_{1},...,p_{n+1}$. Then the altitude $\theta(p_j,W_j^s)$ from $p_j$ to $(n-1)-$face
$W_j^s$ is given by

$$
\cos \theta(p_j,W_j^s) =\displaystyle{\sqrt{1-
\displaystyle\frac{|M|}{M_{jj}}}}, ~~~j=1,...,n+1.
$$

\noindent where $p_{1},...,\hat p_j,...,p_{n+1}$ are vertices of $W_j^s$.
\end{cor}

\end{document}